\documentclass{amsart}
\usepackage{amssymb,latexsym}
\usepackage{amscd,amsthm}

\usepackage[all]{xy}
\usepackage{tikz}
\usetikzlibrary{matrix,arrows}
\usepackage{tikz-cd}

\newtheorem*{theorem*}{Theorem}
\newtheorem{theorem}{Theorem}[section]
\newtheorem{lemma}[theorem]{Lemma}

\newtheorem{proposition}[theorem]{Proposition}
\newtheorem{corollary}[theorem]{Corollary}

\theoremstyle{definition}
\newtheorem{definition}[theorem]{Definition}

\newtheorem*{remark}{Remark}
\newtheorem*{notation}{Notation}
\newtheorem{example}{Example}[subsection]

\DeclareMathOperator{\Ext}{Ext}
\DeclareMathOperator{\Hom}{Hom}

\DeclareMathOperator{\Ker}{Ker}

\newcommand{\cat}[1]{\mathcal{#1}}           

\newcommand{\class}[1]{\mathcal{#1}}   

\newcommand{\Z}{\mathbb{Z}}

\newcommand{\mathcolon}{\colon\,}

\newcommand{\ch}{\textnormal{Ch}(R)}
\newcommand{\cha}[1]{\textnormal{Ch}(\mathcal{#1})}
\newcommand{\rmod}{R\text{-Mod}}

\newcommand{\qcox}{\textnormal{Qco}(\mathbb{X})}

\newcommand{\exinj}{\textnormal{\footnotesize{ex}}{\class{I}}}
\newcommand{\toinj}{\textnormal{\footnotesize{to}}{\class{I}}}

\newcommand{\rightperp}[1]{#1^{\perp}}
\newcommand{\leftperp}[1]{{}^\perp #1}

\begin{document}

\title{Accessibility  and Gorenstein injective envelopes}

\author{Sergio Estrada}
 \address{S.E. \ Universidad de
  Murcia\\ Facultad de Matemáticas \\ Campus de Espinardo \\ Murcia 30100, Spain} \email[Sergio Estrada]{sestrada@um.es}
\urladdr{https://webs.um.es/sestrada/}
\author{James Gillespie}
\address{J.G. \ Ramapo College of New Jersey \\
         School of Theoretical and Applied Science \\
         505 Ramapo Valley Road \\
         Mahwah, NJ 07430\\ U.S.A.}
\email[Jim Gillespie]{jgillesp@ramapo.edu}
\urladdr{http://pages.ramapo.edu/~jgillesp/}

\date{\today}



\thanks{S.E. was partly supported by grant
  22004/PI/22 funded by Fundaci\'on S\'eneca, Agencia de Ciencia y
  Tecnolog\'ia de la Regi\'on de Murcia and by grant PID2024-155576NB-I00 funded by MICIU/ AEI/10.13039/501100011033/FEDER, UE.}


\begin{abstract}
Let $\class G$ be a Grothendieck category. We prove  completeness of the Gorenstein injective cotorsion pair whenever $\cat{G}$  admits a  set of Tate trivial generators, and show that having such generators is necessary for completeness.  In this case it must be a perfect cotorsion pair,   cogenerated by a set,  and equivalent to an injective abelian model structure on $\cat{G}$.
Examples include Grothendieck categories (possibly without enough projectives) that admit a generating set consisting of objects of finite projective dimension, such as the category of quasi-coherent sheaves on a quasi-compact and semi-separated scheme.  More generally, for a given set $\class S$, we characterize the completeness of the Gorenstein $\class B$-injective cotorsion pair, where $\class B=\class S^\perp$, in terms of the existence of a set of $\class B$-Tate trivial generators for $\class G$. The key ingredient to our proof is the fact that any class of the form $\class{B} :=\rightperp{\class{S}}$ is an accessibly embedded, accessible subcategory of $\cat{G}$. 
 The general approach allows for  further applications such as the existence of Ding injective envelopes and other relative Gorenstein injective envelopes without imposing additional assumptions on $\class G$.
\end{abstract}

\maketitle

\section{Introduction}
Let $R$ be an arbitrary ring and $\rmod$ the category of left $R$-modules. An $R$-module $M$ is said to be \emph{Gorenstein injective} if it is equal to some cycle module, $M=Z_nE$, of some exact complex of injectives $E$ having the property that  $\Hom_R(I,E)$ is again an exact  complex for all injective $R$-modules $I$. 
Along with the Gorenstein projective and Gorenstein flat modules, these are the fundamental objects of Gorenstein homological algebra,  a particular variant of relative homological algebra that has been actively researched.

One naturally wonders which results of classical homological algebra have   analogues in Gorenstein homological algebra.
The notion of a Gorenstein injective envelope of an $R$-module is the Gorenstein analogue of the classical notion of an injective envelope. While the existence of injective envelopes was established in 1953 in \cite{eckmann-schopf-injective}, the existence of Gorenstein injective envelopes in $\rmod$ was only recently proved by Šaroch and Šťovíček in~\cite{saroch-stovicek-singular-compactness}.

Module categories are the prototypical examples of Grothendieck categories, and it is well known (for instance, see \cite[Cor.~X.4.3]{stenstrom}) that injective envelopes also exist in any Grothendieck category. As pointed out in \cite{saroch-stovicek-singular-compactness}, their proof of the existence of Gorenstein injective envelopes can be extended to module categories over small preadditive categories, or equivalently, to Grothendieck categories with a generating set consisting of small projective generators. However, the existence of Gorenstein injective envelopes 
in a general Grothendieck category $\cat{G}$ has been an open problem. In the terminology that is typically used in Gorenstein homological algebra, we want to know when the class $\class{GI}$ of all Gorenstein injective objects in $\cat{G}$  is the right half of a perfect cotorsion pair, $(\leftperp{\class{GI}}, \class{GI})$. It means the classes are orthogonal  relative to $\Ext^1_{\cat{G}}(-,-)$, the class $\class{GI}$ is enveloping, and its left orthogonal class $\leftperp{\class{GI}}$ is covering.

The best result so far for general Grothendieck categories was obtained by Krause in \cite[Them.~7.12]{krause-stable-derived}, where he proved the existence of special Gorenstein injective preenvelopes for a locally noetherian Grothendieck category $\cat{G}$ whose derived category is compactly generated. In fact, Krause's proof readily shows that the pair of classes $(\leftperp{\class{GI}}, \class{GI})$ forms a complete cotorsion pair for such categories.  However, nothing has been   known without either the noetherian assumption or without the existence of small projective generators.

In this paper, as a by-product of our more general results, we provide in Corollaries~\ref{cor-Goren-inj} and \ref{cor-Goren-inj-cot} a very satisfactory answer to the above question. We show  that  $(\leftperp{\class{GI}}, \class{GI})$ is a complete (in fact, perfect)  cotorsion pair  in a Grothendieck category $\cat{G}$ if and only if $\cat{G}$ admits a set of \emph{Tate trivial} generators, in the sense of Definition~\ref{definition-B Tate trivial-generators}. This amounts to a  set $\class{U}$  of generators for $\cat{G}$  such that $\class{U}\subseteq \leftperp{\class{GI}}$. Such generating sets exist for  a wide class of Grothendieck categories, including those admitting a family of generators of finite projective dimension. We call such categories \emph{locally finite dimensional Grothendieck categories}; see Section \ref{sec-locally-finite-dim}. In particular, these include Grothendieck categories with a generating set of small projective generators, as well as Grothendieck categories possibly without enough projectives. A main  example is the category of quasi-coherent sheaves on a quasi-compact and semi-separated scheme $\mathbb{X}$ (see Example~\ref{qco}).

In fact, our techniques allow us to prove, simultaneously, and without any  further  assumptions on $\class G$,
the existence of envelopes for a number of relative Gorenstein objects  that have been considered in the literature. These include  Ding injective envelopes, Gorenstein AC-injective envelopes, and Gorenstein $\operatorname{FP}_n$-injective envelopes, to name a few; see \cite{gillespie-iacob-ding-injective-envelopes,bravo-gillespie-hovey, bravo-gillespie-perez-fpn}.  In all cases, we obtain improved general results that hold under the assumption that $\class G$ is a Grothendieck category with a set of generators analogous to the above (see Example~\ref{example-FP_n} for details.)

To state our main result, let
$\class{B} := \rightperp{\class{S}}$ be the right $\Ext$-orthogonal class to any given set $\class{S}$ of objects in $\cat{G}$. An object $M\in\cat{G}$  is said to be \emph{Gorenstein $\class{B}$-injective} if it is equal to some cycle object, $M=Z_nE$, of some exact complex of injective objects $E$ having the property that  $\Hom_{\cat{G}}(B,E)$ is an exact  complex of abelian groups for all $B\in\class{B}$. Let $\class{GI}_{\class{B}}$ denote the class of all Gorenstein $\class{B}$-injective objects.

\begin{theorem*}(Corollary~\ref{cor-Goren-inj-cot}) Let $\class G$ be a Grothendieck category with a set of generators of finite projective dimension, and let $\class S$ be any set of objects. Let $\class{B} := \rightperp{\class{S}}$.  Then $(\leftperp{\class{GI}_{\class{B}}}, \class{GI}_{\class{B}})$ is a  perfect cotorsion pair. That is,  the class $\leftperp{\class{GI}_{\class{B}}}$ is covering and the class $\class{GI}_{\class{B}}$ is enveloping.
\end{theorem*}

It is well known in this field that there is a close  relationship between covers and envelopes, cotorsion pairs, and Quillen model structures on additive categories. Theorem~\ref{them-Goren-inj-cot} gives the model category interpretation of the above result. The theorem lists numerous desirable properties of the corresponding \emph{Gorenstein $\class{B}$-injective model structure} and its  homotopy category. For one, it sends all objects of $\class{B}$ to 0 in the homotopy category, and it is the best injective model structure on $\cat{G}$ with this property.

It is worth noting that our proofs are entirely different than the ones given   in~\cite{saroch-stovicek-singular-compactness} for $R$-modules.
In particular, the proofs of Theorem~\ref{them-Goren-inj-cot} and Corollary~\ref{cor-Goren-inj-cot} are based on two key results. The first concerns the  accessibility of right cotorsion classes, that is, classes of the form $\class{B} :=\rightperp{\class{S}}$ for a given set of objects $\class{S}$. By Theorem~\ref{them-rightperp-accessible}, any such class $\class{B}$ posesses a set $T$ for which each $B\in\class{B}$ is a direct limit of objects from $T$. This is the central observation of this work and the key new ingredient to our proofs. We prove  Theorem~\ref{them-rightperp-accessible}  in a way that makes it clear that it is a specialization of a general result by Ad\'ameck and Rosick\'y \cite[Prop.~4.7]{adamek-rosicky}. In their terminology,  $\class{B} :=\rightperp{\class{S}}$ is an \emph{accessibly embedded, accessible subcategory} of $\cat{G}$. This useful fact does not seem to have  been recognized in homological algebra, or at least not exploited to its fullest.   It allows us to go on and prove our main results using only standard tools from set-theoretic homological algebra.
In particular, we combine the accessibility result with our Theorem~\ref{them-A-acyclic-model}, which in turn goes back  to a result developed by Bravo, Gillespie and Hovey  in \cite[Them.~4.1]{bravo-gillespie-hovey}.

\subsubsection*{Notation, Terminology, and Preliminaries}
Given an additive category $\cat{A}$, we denote by $\cha{A}$ the category of chain complexes over $\cat{A}$. We use homological notation for complexes, meaning that our differentials lower degree. For an object $M\in\cat{A}$, we denote by $S^n(M)$ the chain complex consisting of $M$ in degree $n$ and 0 everywhere else. We call $S^n(M)$ the \emph{$n$-sphere} on $M$. We also have the \emph{$n$-disk} on $M$, denoted by $D^n(M)$. It is the complex with $M$ concentrated in degrees $n$ and $n-1$, with connecting differential $1_M$, and 0 elsewhere.

A standard reference for Gorenstein homological algebra in categories of modules over a ring is the book~\cite{enochs-jenda-book}. For basic results on Grothendieck categories we will refer to the book~\cite{stenstrom}. The cornerstone result of Theorem~\ref{them-rightperp-accessible} heavily relies on the theory of locally presentable and accessible categories as laid out in~\cite{adamek-rosicky}. The recent book~\cite{gillespie-book} will also be heavily used to reference results on cotorsion pairs and their relation to abelian model categories. We note now that what we call here an \emph{injective cotorsion pair}, denoted by say $(\class{W},\class{R})$, is equivalent to an \emph{injective model structure}, $(All, \class{W},\class{R})$. See~\cite[\S2.5 and~\S8.1]{gillespie-book}.

\section{Accessible categories and right cotorsion classes}\label{access-rightperp}

The theory of locally presentable and accessible categories is a convenient setting to obtain far reaching results in 
 homological algebra.
We start by recalling some of the basic terminology of the theory, following the standard reference~\cite{adamek-rosicky}.

Recall that a \emph{regular} cardinal is an infinite cardinal which is not the sum of a smaller
number of smaller cardinals. All infinite successor cardinals are regular.  But  $\aleph_{\omega} =
\sum_{n < \omega} \aleph_n$ is not a regular cardinal.

\begin{definition}
Let $\cat{K}$ be a category and $\lambda$ a regular cardinal.
By a $\lambda$-direct limit, denoted $\varinjlim M_i$,  we mean the colimit of a directed system $\{M_i, f_{ij}\}$ indexed by some $\lambda$-directed poset.   We say that $\cat{K}$ \emph{has $\lambda$-direct limits} if colimits of such diagrams always exist.
\begin{enumerate}
\item  An object $M\in\cat{K}$ is called \emph{$\lambda$-presentable} if $\Hom_{\cat{K}}(M,-)$ preserves $\lambda$-direct limits.
\item We say that $\class{K}$ is \emph{$\lambda$-accessible}  if it has $\lambda$-direct limits and possesses a set $T$ of $\lambda$-presentable objects such that every object of $\class{K}$ is a $\lambda$-direct limit of objects from $T$. By an \emph{accessible} category we mean one that is $\lambda$-accessible for some regular $\lambda$.
\item We say that $\class{K}$ is \emph{locally $\lambda$-presentable}  if it is cocomplete and $\lambda$-accessible. By a \emph{locally presentable} category we mean a cocomplete accessible category.
\end{enumerate}
\end{definition}

Any Grothendieck category is locally presentable~\cite[Prop.~3.10]{beke}.
In particular, the category $\rmod$ of all (say left) modules over a ring $R$ is a locally $\omega$-presentable category. Such a category is called a \emph{locally finitely presentable category}.  Similarly one speaks of a  \emph{finitely accessible category}. For example, by the Govorov-Lazard Theorem, the full subcategory of $\rmod$ consisting of all flat modules is finitely accessible.

Our  goal is to prove Theorem~\ref{them-rightperp-accessible}, which roughly says that the right Ext-orthogonal class to a set of objects in an accessible Quillen exact category must itself be an accessible full subcategory. The next few definitions and lemmas are meant to illuminate some of the language used to state the theorem.

\begin{definition}\cite[Def.~2.35]{adamek-rosicky}\label{def-accessibly-embedded}
Let $\cat{K}$ be an accessible category.
A  subcategory $\class{B}\subseteq \cat{K}$ is said to be   \emph{accessibly embedded}  if it is full and there exists a regular cardinal $\lambda$ such that $\class{B}$ is closed under $\lambda$-direct limits in $\cat{K}$. For any particular such $\lambda$, we will also say that $\class{B} \subseteq \cat{K}$ is  \emph{$\lambda$-accessibly embedded}.
\end{definition}
Note that a $\lambda$-accessibly embedded  $\cat{B} \subseteq \cat{K}$ has $\lambda$-direct limits  which are preserved by the inclusion functor. The argument for this easy fact appears below in the proof of Lemma~\ref{lemma-acc-embed}.
Also, any accessibly embedded $\cat{B}$ must be isomorphism closed and so is a \emph{strictly full}  subcategory of $\cat{K}$. 
 This follows from considering colimits of functors $\{0\} \xrightarrow{} \cat{K}$, because the one point set $\{0\}$ is trivially a $\lambda$-directed poset.

The next two lemmas do not seem to be explicitly stated in~\cite{adamek-rosicky}, though  they are  used there implicitly.

\begin{lemma}\label{lemma-acc-embed}
Let $\lambda$ be a regular cardinal. Assume $\cat{K}$ is a $\lambda$-accessible category, and   $\class{B} \subseteq \cat{K}$ is a strictly full subcategory with $\class{B}$ itself being $\lambda$-accessible. Then the following are equivalent:
\begin{enumerate}
\item $\class{B} \subseteq \cat{K}$ is $\lambda$-accessibly embedded. 
\item The inclusion $\cat{B} \hookrightarrow \cat{K}$  is a $\lambda$-accessible functor in the sense of~\cite[Def.~2.16]{adamek-rosicky}. Here, this just means it preserves $\lambda$-direct limits.
\end{enumerate}
When these conditions hold, we will say that  $\class{B} \subseteq \cat{K}$ is a \textbf{$\boldsymbol{\lambda}$-accessibly embedded, $\boldsymbol{\lambda}$-accessible subcategory}.  We say that $\class{B} \subseteq \cat{K}$  is an \textbf{accessibly embedded, accessible subcategory} if  such  a regular $\lambda$ exists.
\end{lemma}

For example, the full subcategory of $\rmod$ consisting of all flat $R$-modules is an $\omega$-accessibly embedded, $\omega$-accessible subcategory.  Note that if $\class{F}$ is any skeleton for the category of flat $R$-modules,  then the  inclusion $\class{F}  \hookrightarrow \rmod$ is still an accessible functor. But since  $\class{F}\subseteq\rmod$  is not strictly full, it  is also not accessibly embedded.

\begin{proof}
(1) $\impliedby$ (2).  Assume that the inclusion $\cat{B} \hookrightarrow \cat{K}$ is a $\lambda$-accessible functor.  Given any $\lambda$-directed system $\{B_i, f_{ij}\}$ of morphisms between objects of $\cat{B}$, we are asserting that any colimit for it in $\cat{K}$ is in fact in $\cat{B}$. But a colimit cone $(\varinjlim B_i, \{\eta_i\})$   does exist in $\cat{B}$, and by  hypothesis it also serves as a  colimit cone in $\cat{K}$. Any other colimit in $\cat{K}$ for $\{B_i, f_{ij}\}$ must be isomorphic to $\varinjlim B_i\in\cat{B}$. So the strictly full assumption implies that \emph{any}  colimit of  $\{B_i, f_{ij}\}$ in $\cat{K}$ must be in $\cat{B}$.

\noindent  (1) $\implies$ (2).  Assume $\class{B}$ is closed under $\lambda$-direct limits in $\cat{K}$. We are to show that the inclusion preserves $\lambda$-direct limits. So again let $\{B_i, f_{ij}\}$  be a $\lambda$-directed system of morphisms between objects in $\cat{B}$.  Let $(\varinjlim B_i, \{\eta_i\})$ be a  colimit cone for it, but in $\cat{K}$.
By hypothesis,  $\varinjlim B_i \in \cat{B}$. Since the subcategory is full, each  $\eta_i \colon B_i \xrightarrow{} \varinjlim B_i$ is also in $\class{B}$.  Likewise, for any  `imposter' cocone in $\cat{B}$,  the unique map $h \colon \varinjlim B_i \xrightarrow{} B$  in $\cat{K}$ satisfying the universal property must also be in $\cat{B}$, by the fullness. It follows that $(\varinjlim B_i, \{\eta_i\})$ is indeed the  colimit cone for the system $\{B_i, f_{ij}\}$, in $\cat{B}$, and that the inclusion preserves the colimit cone.
\end{proof}

The notion of a \emph{sharply greater} regular cardinal $\mu \vartriangleright \lambda$ is found in~\cite[Def.~2.12]{adamek-rosicky}. The point of them is that whenever $\cat{K}$ is $\lambda$-accessible then it is also $\mu$-accessible whenever $\mu$ is sharply greater than $\lambda$.

\begin{lemma}\label{lemma-raising-acc-index}
Let $\lambda$ be a regular cardinal.  Assume  that  $\class{B} \subseteq \cat{K}$ is a $\lambda$-accessibly embedded, $\lambda$-accessible subcategory.
\begin{enumerate}
\item  For all regular cardinals  $\lambda' \vartriangleright \lambda$, we again have that  $\class{B} \subseteq \cat{K}$ is a $\lambda'$-accessibly embedded, $\lambda'$-accessible subcategory.
\item  In addition, there exist (arbitrarily large)  regular cardinals  $\mu \vartriangleright \lambda$ such that the $\mu$-presentable objects internal to  the accessible category $\class{B}$ are precisely the objects of $\cat{B}$ that are $\mu$-presentable in the ambient accessible category $\cat{K}$.
\end{enumerate}
\end{lemma}

\begin{proof}
Statement (1) follows from~\cite[Them.~2.11(i)/Def.~2.12]{adamek-rosicky} and ~\cite[Remarks~2.18(2)]{adamek-rosicky}.

We prove (2).  We first observe that if $B\in\class{B}$ is  $\lambda$-presentable relative to  $\cat{K}$, then $B$ is also $\lambda$-presentable relative to the accessible category $\cat{B}$.   Indeed the former means that the functor $\Hom_{\cat{K}}(B, -)  \colon \cat{K} \xrightarrow{} \textbf{Set}$ preserves $\lambda$-direct limits.
  Since the restriction of this functor to $\cat{B}$ is exactly $\Hom_{\cat{B}}(B, -)  \colon \cat{B} \xrightarrow{} \textbf{Set}$, it follows from Lemma~\ref{lemma-acc-embed}(2) that $B$ is $\lambda$-presentable relative to $\cat{B}$.  So we have shown  $\lambda$-presentablility in $\cat{K}$ implies $\lambda$-presentability in  $\cat{B}$. So by combining  this with statement (1), we have that for any  regular cardinal  $\lambda' \vartriangleright \lambda$:  If $B\in\class{B}$ is  $\lambda'$-presentable relative to  $\cat{K}$, then $B$ is also $\lambda'$-presentable relative to   $\cat{B}$.
  On the other hand, it  follows from~\cite[Uniformization Theorem~2.19]{adamek-rosicky} that there  exist  (arbitrarily large)  regular cardinals $\mu \vartriangleright \lambda$ such that, conversely,  the inclusion functor takes $\mu$-presentable objects of $\cat{B}$ to $\mu$-presentable objects of $\cat{K}$. (The proof constructs a particular regular $\lambda' \vartriangleright \lambda$ such that this is the case for  any other regular $\mu \vartriangleright \lambda'$.)
\end{proof}

The following general statement is in the language of Quillen exact category structures in the sense of~\cite{quillen-algebraic K-theory} and~\cite{buhler-exact categories}. These are also discussed throughout~\cite{gillespie-book}, and in particular one can find the terminology of \emph{admissible generators}   in~\cite[\S9.5]{gillespie-book}.\footnote{We note  that any accessible additive category is  idempotent complete~\cite[Prop.~12.3(4)]{gillespie-book}.}

\begin{theorem}\label{them-rightperp-accessible}
Assume $\cat{K}$ is an accessible  additive category with coproducts and   $(\cat{K},\class{E})$ is a Quillen exact  structure possessing  a set of admissible generators.
Given any set of objects $\class{S}$,  set  $\class{B} := \rightperp{\class{S}}$. That is, $\cat{B}$ is the class of all objects $B \in \cat{K}$ such that $\Ext_{\class{E}}^1(S,B) = 0$ for all $S\in\class{S}$.  Then  the strictly full subcategory $\class{B}\subseteq \cat{K}$ is an accessibly embedded, accessible subcategory.  This implies the existence of (arbitrarily large) regular cardinals $\mu$ such that both of these hold:
\begin{enumerate}
\item $\class{B}$ and $\cat{K}$ are each $\mu$-accessible categories  and the  inclusion functor $\cat{B} \hookrightarrow \cat{K}$ preserves $\mu$-direct limits.
\item The internal $\mu$-presentable objects of the accessible category $\class{B}$ are precisely the objects of $\cat{B}$ that are $\mu$-presentable in the ambient accessible  category $\cat{K}$.
\end{enumerate}
In particular, $\cat{B}$ is closed under $\mu$-direct limits and there is a set  (not a proper class) of $\mu$-presentable objects $T\subseteq \cat{B}$ such that each $B\in\cat{B}$ is a $\mu$-direct limit of objects from $T$.
\end{theorem}

Our initial  proof of Theorem~\ref{them-rightperp-accessible} was based on the purity argument of~\cite[Lemma~3]{krause-approximations and adjoints}, and then applying  the powerful~\cite[Cor.~2.36]{adamek-rosicky}. The latter says that if $\cat{K}$ is an accessible category and $\class{B}\subseteq \cat{K}$ is accessibly embedded, then $\class{B}$ is itself an accessible category if and only if there exists a regular cardinal $\alpha$ for which $\class{B}$ is closed,  in $\cat{K}$, under taking $\alpha$-pure subobjects.

\begin{proof}
This proof follows the one in~\cite[Prop.~4.7]{adamek-rosicky}, combined with  the existence of sets of generating monomorphisms as shown in~\cite{saorin-stovicek}. The latter result can also be found in~\cite[\S9.6]{gillespie-book}. Indeed by~\cite[Prop.~9.29]{gillespie-book} there exists a set  $I_{\class{S}} =\{k : K \rightarrowtail X\}$ of generating monics for $\class{B} := \rightperp{\class{S}}$. This means that $Y \in \class{B}$ if and only if for any $k \colon K \rightarrowtail X$ in $I_{\class{S}}$ and  any morphism  $f \colon K \xrightarrow{} Y$, there exists an extension of $f$ along $k$.  Now if we take any regular cardinal $\lambda$ such that the domain of each $k \in I_{\class{S}}$ is $\lambda$-presentable, then $\class{B} = \rightperp{\class{S}}$ is closed under $\lambda$-direct limits; see~\cite[Lemma~12.5]{gillespie-book}. In other words, the full subcategory $\cat{B} \subseteq \cat{K}$ is  $\lambda$-accessibly embedded in the sense of Definition~\ref{def-accessibly-embedded}.
Note that we are free to choose our regular $\lambda$ so that both the domain and the codomain of each $k \in I_{\class{S}}$ is $\lambda$-presentable and also so that $\cat{K}$ is  $\lambda$-accessible. Then again $\cat{B} \subseteq \cat{K}$ is  $\lambda$-accessibly embedded and we now can easily argue that $\cat{B}$ is closed under $\lambda$-pure subobjects (and this is analogous to the proof  of~\cite[Lemma~3]{krause-approximations and adjoints}).  Indeed assume $p \colon P \rightarrowtail B$ is a $\lambda$-pure (in $\cat{K}$) subobject of some $B\in \class{B}$. We wish to show $P \in \class{B}$, and for this it is enough to show that   for any $k \colon K \rightarrowtail X$ in $I_{\class{S}}$ and  any given morphism  $f \colon K \xrightarrow{} P$, there exists an extension of $f$ along $k$. But given such an $f$, since $B \in \class{B}$, we do have that the composite $pf$ admits such an extension, call it $g$. Diagrammatically, the solid arrows provide a commutative square in $\cat{K}$
$$\begin{tikzcd}[ampersand replacement=\&]
K \arrow[r, rightarrowtail, "k"] \arrow[d, "f" '] \& X \arrow[dl,  , dashed, "\exists" ']  \arrow[d,  "g"] \\
P  \arrow[r, rightarrowtail, "p"] \& B.
\end{tikzcd}$$
By the very definition of $p$ being $\lambda$-pure, a dashed arrow $X \xrightarrow{} P$ exists so that the upper left triangle commutes. That is, $f$ extends over $k$, proving   $P \in \class{B}$.
It  follows from the aforementioned~\cite[Cor.~2.36]{adamek-rosicky}  that $\class{B} \subseteq \cat{K}$ is an accessibly embedded, accessible subcategory.
In particular, the proof there shows that there is  a regular $\lambda' \vartriangleright \lambda$  such that  $\class{B}$ is $\lambda'$-accessible ($\lambda'$ is called $\mu$ in the proof there). Note then that $\cat{K}$ too is  $\lambda'$-accessible, and still, $\class{B}\subseteq \cat{K}$ is $\lambda'$-accessibly embedded. It follows from Lemma~\ref{lemma-raising-acc-index}(2) that we can  raise the regular cardinal  $\lambda'$ to   arbitrarily  large ones  $\mu \vartriangleright \lambda'$ such that the statement in (2) holds. For such regular cardinals $\mu$, we have that $\class{B}\subseteq \cat{K}$ is again a $\mu$-accessibly embedded, $\mu$-accessible subcategory and the proof is complete.
\end{proof}

\begin{corollary}\label{cor-rightperp-accessible}
Let $\cat{G}$ be any Grothendieck category and let  $(\class{A}, \class{B})$ be any cotorsion pair that is cogenerated by a set.  Then there exists a regular cardinal $\mu$ 
 such that $\cat{B}$ is closed under $\mu$-direct limits, and there is  a set (not a proper class) $T\subseteq \class{B}$ of $\mu$-presentable objects such that each $B\in\cat{B}$ is a $\mu$-direct limit of objects from $T$.
\end{corollary}

The corollary below follows from the version of Baer's criterion for injectivity that holds in Grothendieck categories; \cite[Prop.~V.2.9]{stenstrom}.

\begin{corollary}\label{cor-injectives-accessible}
For any Grothendieck category $\cat{G}$ there exists a regular cardinal $\mu$ such that  $\mu$-direct limits of injectives are again injective, and there is a set (not a proper class)  $T$ of $\mu$-presentable injective objects  such that every injective object is a $\mu$-direct limit of objects from $T$.
\end{corollary}

\section{$A$-acyclic complexes of injectives and model structure}

Throughout this section we let $\cat{G}$ be a Grothendieck category with a set of generators   $\mathcal{U} = \{U_i\}$. We prove Theorem~\ref{them-A-acyclic-model}, which is useful for  constructing  injective abelian  model structures on the category $\cha{G}$ of chain complexes.

\begin{definition}\label{def-A-acyclic}
Let $X$ be a chain complex of injective objects from $\cat{G}$.
\begin{itemize}
\item  Let $A\in\cat{G}$ be an object. We say that $X$ is  \emph{$A$-acyclic} if $\Hom_{\cat{G}}(A,X)$ is an exact complex of abelian groups.  If $X$ itself is also exact, we say that it is an  \emph{exact $A$-acyclic} complex of injectives.
\item Let $\class{A}\subseteq \cat{G}$ be a class of objects. We say that $X$ is \emph{(exact) $\class{A}$-acyclic} if it is (exact) $A$-acyclic for all $A\in\class{A}$.
\end{itemize}
\end{definition}

\begin{definition}
$A = \{A_j\} \subseteq \cat{G}$ be any set of objects. We let $\varinjlim A$ denote the \emph{direct limit closure} of $A$. By this we mean the class of all objects equal to the colimit of some directed system of objects from $A$.
\end{definition}

The root of the theorem below goes back to~\cite[Them.~4.1]{bravo-gillespie-hovey}. But the statements here concerning direct limits were not there and will prove extremely useful when combined with Theorem~\ref{them-rightperp-accessible}.

\begin{notation}
Given any class of objects $\class{B}\subseteq \cat{G}$, we let $\class{I}_{\class{B}}$ denote the class of all (not necessarily exact) $\class{B}$-acyclic chain complexes of injectives.
\end{notation}

\begin{theorem}\label{them-A-acyclic-model}
Let $A = \{A_j\} \subseteq \cat{G}$ be any set of objects and let   $\class{B}$ be any class of objects  such that $A \subseteq \class{B} \subseteq \varinjlim A$.  Then we have $\class{I}_{\class{B}} = \class{I}_A$ and there is a cofibrantly generated injective   model structure $$\mathfrak{M}_A = (All, \class{W}, \class{I}_{\class{B}})$$ on $\cha{G}$ with each of the following properties:
\begin{enumerate}
\item  
The thick class $\class{W}$ contains all contractible complexes and  it is  closed under direct limits.
\item An object $M$ of the ground category $\cat{G}$ is in $\class{W}$, meaning $S^0(M) \in \class{W}$,  if and only if, $\Hom_{\cat{G}}(M,X)$ is an exact complex of abelian groups  for all $X \in \class{I}_A$.  In other words, $M$ is trivial precisely   when each  $X\in \class{I}_A$ is $M$-acyclic. \\  In particular, $A \subseteq \class{B} \subseteq \varinjlim A\subseteq \class{W}$.
\item More generally, a  complex  $W$ is  in $\class{W}$, if and only if, the total hom complex $\Hom_{\cat{G}}(W,X)$ is exact for all $X \in \class{I}_A$.  
\item  $\class{W}$  is also thick when viewed as a full subcategory of $K(\cat{G})$, the chain homotopy category of $\cat{G}$. A chain map is a weak equivalence if and only if  its mapping cone is in $\class{W}$. The formal homotopy category $\textnormal{Ho}(\mathfrak{M}_A)$  identifies with the Verdier quotient  $K(\cat{G})/\class{W}$, and
 it is equivalent to the strictly full subcategory of  $K(\cat{G})$ generated  by the complexes in  $\class{I}_A$. Lastly, this is a well-generated triangulated category.
\end{enumerate}
\end{theorem}

\begin{proof}
By Baer's criterion, the set of all $\{\,U_i/C\,\}$ where $C$ ranges through all subobjects of each $U_i$, cogenerates the categorical injective cotorsion pair, $(\cat{G},\class{I})$,  in our ground category $\cat{G}$. We define a set of complexes $\class{S}$ as follows where  all of the $n$'s  range through $\Z$:
\[
\class{S} = \{D^{n} (U_i/C) \,|\, U_i \in\mathcal{U} , \ \forall C \subseteq U_i\} \bigcup   \{S^{n} (A_j) \,|\, A_j \in A \}.
\]
One checks that $\rightperp{\class{S}} = \class{I}_A$.
Indeed by~\cite[Prop.~4.4]{gillespie-degreewise-model-strucs}, the right Ext-orthogonal to the set $$\{D^{n} (U_i/C) \,|\, U_i \in\mathcal{U} , \ \forall C \subseteq U_i\}$$ is precisely the class of all  complexes of injectives.  So then $\rightperp{\class{S}}$ is precisely the class of all complexes of injectives $X$ such that $$\Ext^{1}_{\cha{G}}(S^{n}(A_j),X) = 0$$
 for  all $n$ and all $A_j\in A$.
 But since $X$ is a complex of injectives, one can show (see  \cite[Cor.~10.22 and Exer.~10.4.1]{gillespie-book}) that  for a $\cat{G}$-object $M$ we have an isomorphism
\begin{equation}\label{eq-iso}
  \tag{$\dagger$}
\Ext^{1}_{\cha{G}}(S^{n}(M),X) \cong H_{n-1}[\Hom_{\cat{G}}(M,X)].
 \end{equation}
From this it follows that $\rightperp{\class{S}} = \class{I}_A$.

Now defining $\class{W} := \leftperp{\class{I}_A}$, we have  that $(\class{W}, \class{I}_A)$ is a cotorsion pair that is cogenerated by $\class{S}$. We note that both $\class{W}$ and $\class{I}_A$ are closed under suspensions (shifts).  The isomorphism in~(\ref{eq-iso}) is a special case of the well-known isomorphism
$$\Ext^{1}_{dw}(\Sigma^{n}W,X) \cong H_{n-1}[\Hom_{\cat{G}}(W,X)],$$
between degreewise split extensions of complexes and the homology of the total hom complex.
(For example, see again  \cite[Cor.~10.22]{gillespie-book}.) The characterization of $\class{W}$ stated in (3)  follows from this isomorphism. What we have asserted in (2) follows from the special case isomorphism shown above in~(\ref{eq-iso}).

It follows easily from property (3)  that $\class{W}$ is closed under direct summands and satisfies the 2 out of 3 property on short exact sequences. It is also clear that $\class{W}$ contains all contractible complexes. This includes  the categorically injective complexes, which are precisely the contractible complexes of injectives, as well as the generating set $\{D^n(U_i) \,|\, U_i\in\mathcal{U}\}$ for the Grothendieck category $\cha{G}$. From these two facts it follows  that  $(\class{W}, \class{I}_A)$ is a (cofibrantly generated and functorially complete) injective cotorsion pair; see~\cite[Prop.~2.21(4) and Them.~9.34]{gillespie-book}. It is then automatic from~\cite[Prop.~3.2]{gillespie-ding-modules} that $\class{W}$ is closed under direct limits. (This can also be found in~\cite[Cor.~9.57]{gillespie-book}.)

At this point we have constructed a cofibrantly generated abelian model structure $$\mathfrak{M} = (All, \class{W}, \class{I}_A)$$ on $\cha{G}$ having all of the properties listed in (1), (2),  and (3). The   properties listed in (4) follow automatically from general statements; see~\cite[Them.~10.27 and~10.29 and Cor.~12.24]{gillespie-book}.

To finish, let us make  clear that  $\class{I}_{\class{B}} = \class{I}_A$.
Because of the containments $$A \subseteq \class{B} \subseteq \varinjlim A,$$ we clearly have that for any complex $X$  of injectives,
$$\textnormal{$X$ is  $(\varinjlim A)$-acyclic $\implies$ $X$ is  $\class{B}$-acyclic $\implies$  $X$ is  $A$-acyclic.}$$
  But it follows immediately from the properties stated in (1) and (2)   that
  $$\textnormal{$X$ is $A$-acyclic $\implies$ $X$ is  $(\varinjlim A)$-acyclic.}$$  So the three types of acyclicity coincide. In particular, $\class{I}_A = \class{I}_{\class{B}}$.
\end{proof}

The following result is immediate by combining Corollary~\ref{cor-rightperp-accessible} and Theorem~\ref{them-A-acyclic-model}.

\begin{corollary}\label{cor-injectives-B-acyclic}
Let $\cat{G}$ be a Grothendieck category and $(\class{A}, \class{B})$ be any cotorsion pair that is cogenerated by a set.   Then there exists a set of objects $T\subseteq \class{B}$ such that
 $\class{I}_T = \class{I}_{\class{B}}$ and for which applying Theorem~\ref{them-A-acyclic-model} to $T$ yields exactly $\mathfrak{M}_T = (All, \class{W}, \class{I}_{\class{B}})$. That is, the fibrant objects are precisely the (not necessarily exact) $\class{B}$-acyclic complexes of injectives.
\end{corollary}

In particular, by Corollary~\ref{cor-injectives-accessible} we get the following curious result.

\begin{corollary}\label{cor-injectives-acyclic}
Let $\class{I}$ denote the class of all injective objects in any Grothendieck category $\cat{G}$.  Allow $\class{I}_{\class{I}}$ to denote  the class of all (not necessarily exact)  $\class{I}$-acyclic complexes of injectives $X$.  Then there exists a set of injective objects $T$ such that
 $\class{I}_T = \class{I}_{\class{I}}$ and for which applying Theorem~\ref{them-A-acyclic-model} to $T$ yields exactly $\mathfrak{M}_T = (All, \class{W}, \class{I}_{\class{I}})$.
\end{corollary}

\begin{remark}
Although we will not use it in this paper, it seems worth pointing out that Theorem~\ref{them-A-acyclic-model} is in fact a corollary of a more general theorem. The reader will note that  the notion of an $A$-acyclic complex $X$ of injectives (Definition~\ref{def-A-acyclic})  applies equally well to any set $A = \{A_j\}$ of \emph{chain complexes} over $\cat{G}$, not just $\cat{G}$-objects.  Indeed one just uses acyclicity of the total hom complex $\Hom_{\cat{G}}(A_j,X)$ for complexes $A_j \in A$.  Then the analogous statement of Theorem~\ref{them-A-acyclic-model}, and its proof go through in the same way. One just adjusts the cogenerating set  in the proof to be the more general  set \[
\class{S} = \{D^{n} (U_i/C) \,|\, U_i \in\mathcal{U} , \ \forall C \subseteq U_i\} \bigcup   \{\Sigma^{n}A_j \,|\, A_j \in A \}.
\]
The rest of the proof goes through in the same way.
\end{remark}

\section{Generators for totally acyclic complexes of injectives}\label{section-finite-dim}

Assume again throughout that $\cat{G}$ is a Grothendieck category. 

\begin{notation}
Given any class of objects $\class{B}\subseteq \cat{G}$, we let $\exinj_{\class{B}}$ denote the class of all \emph{exact} $\class{B}$-acyclic chain complexes of injectives. (See Definition~\ref{def-A-acyclic}.)
\end{notation}

We will now give a result providing conditions  guaranteeing that $\exinj_{\class{B}}$ is the class of fibrant objects in an abelian model structure on $\cha{G}$. A technical condition we assume is for $\cat{G}$ to admit a set $\mathcal{U}$ of generators such that each complex $X \in  \exinj_{\class{B}}$ is $\class{U}$-acyclic (again in the sense of Definition~\ref{def-A-acyclic}). But there are two important points to be made. First, the condition is \emph{necessary} (see Lemma~\ref{lemma-generators-cond}) if we are to obtain completeness of Gorenstein $\class{B}$-injective cotorsion pairs. Second, as we will point out at the end of this section, this condition holds automatically for a vast class of  important Grothendieck categories; the ones we call \emph{locally finite dimensional}.

\begin{proposition}\label{them-exact-B-acyclic-models}
Let $A = \{A_j\} \subseteq \cat{G}$ be any set of objects and let   $\class{B}$ be any class of objects  such that $A \subseteq \class{B} \subseteq \varinjlim A$. Assume $\cat{G}$ admits a set $\mathcal{U}$ of generators such that each $X \in  \exinj_{\class{B}}$ is $\class{U}$-acyclic.  Set $\widehat{A} = A\cup\mathcal{U}$.
Then $\class{I}_{\widehat{A}} = \exinj_{\class{B}}$ and applying Theorem~\ref{them-A-acyclic-model} to the set $\widehat{A}$ yields a cofibrantly generated abelian model structure
$$\mathfrak{M}_{\widehat{A}} = (All, \class{W}, \exinj_{\class{B}})$$
with all of the properties (1)--(4) listed there. 
\end{proposition}

\begin{proof}
Of course by taking the set $A$ to be $\widehat{A}$ in  Theorem~\ref{them-A-acyclic-model} we do get a model structure, $\mathfrak{M}_{\widehat{A}} = (All, \class{W}, \class{I}_{\widehat{A}})$, with all of the properties listed there. We only need to show that $\class{I}_{\widehat{A}} =  \exinj_{\class{B}}$, the class of all exact $\class{B}$-acyclic complexes of injectives.

\noindent $(\subseteq)$
Assume  $X \in \class{I}_{\widehat{A}}$. Then, by definition, $X$ is a complex of injectives that is both $A$-acyclic and $\class{U}$-acyclic. By Theorem~\ref{them-A-acyclic-model}, $\class{I}_{\class{B}} = \class{I}_A$, so $X$ is $\class{B}$-acyclic.  Being also $\class{U}$-acyclic  implies that $X$ itself is exact, as $\class{U}$ is a set of generators.

\noindent $(\supseteq)$ On the other hand, if $X\in\exinj_{\class{B}}$, then certainly it is $A$-acyclic, as $A\subseteq \class{B}$. By hypothesis, $X$ is also $\class{U}$-acyclic. This means $X$ is in $\class{I}_{\widehat{A}}$. \\
 This shows  $\class{I}_{\widehat{A}} =  \exinj_{\class{B}}$, and that is all that was needed to be shown.
\end{proof}

Thus we have a version of Corollary~\ref{cor-injectives-B-acyclic} where the fibrant objects are the \emph{exact} $\class{B}$-acyclic complexes.
We will call the model structure in the next corollary the \textbf{\emph{exact $\class{B}$-acyclic model structure}}.

\begin{corollary}\label{cor-exact-injectives-B-acyclic}
Let $\cat{G}$ be a  Grothendieck category and $(\class{A}, \class{B})$ be a cotorsion pair that is cogenerated by a set.  Assume $\cat{G}$ admits a set $\mathcal{U}$ of generators such that each $X \in  \exinj_{\class{B}}$ is $\class{U}$-acyclic.  Then there exists a set of objects $\widehat{T}$ for which $\class{I}_{\widehat{T}} = \exinj_{\class{B}}$ and applying Theorem~\ref{them-A-acyclic-model} to the set $\widehat{T}$ yields a cofibrantly generated abelian model structure
$$\mathfrak{M}_{\widehat{T}} = (All, \class{W}, \exinj_{\class{B}})$$
with all of the properties (1)--(4) listed there. 
\end{corollary}

\begin{proof}
Precisely, let $T$ be a set as guaranteed by Corollary~\ref{cor-rightperp-accessible}.  Then as in Proposition~\ref{them-exact-B-acyclic-models}, set $\widehat{T} = T\cup\mathcal{U}$, and apply the theorem.
\end{proof}

Similarly, we get the exact version of Corollary~\ref{cor-injectives-acyclic}. The usual terminology for an exact $\class{I}$-acyclic complex of injectives is that of a \emph{totally acyclic complex of injectives}. We will use $\toinj$ (instead of writing $\exinj_{\class{I}}$) to denote the class of all such complexes.

\begin{corollary}\label{cor-injectives-totally-acyclic}
Let $\cat{G}$ be a Grothendieck category and let  $\toinj$ denote the class of all totally acyclic complexes of injectives.  Assume $\cat{G}$ admits a set $\mathcal{U}$ of generators such that each $X \in  \toinj_{\class{B}}$ is $\class{U}$-acyclic.
Then there exists a set of objects $\widehat{T}$ for which $\class{I}_{\widehat{T}} = \toinj$ and applying Theorem~\ref{them-A-acyclic-model} to the set $\widehat{T}$ yields a cofibrantly generated abelian model structure
$$\mathfrak{M}_{\widehat{T}} = (All, \class{W}, \toinj)$$
with all of the properties (1)--(4) listed there. 
\end{corollary}

\begin{proof}
Precisely, let $T$ be a set as guaranteed by  Corollary~\ref{cor-injectives-accessible}. Then as in Proposition~\ref{them-exact-B-acyclic-models}, set $\widehat{T} = T\cup\mathcal{U}$, and apply the theorem.
\end{proof}

\subsection{Locally finite dimensional Grothendieck categories}\label{sec-locally-finite-dim}
Following the approach  in~\cite[\S5.1]{gillespie-models-for-hocats-of-injectives}, there is an abundance of important Grothendieck categories admitting a set $\class{U}$ of generators for which it is automatic that every $X \in  \exinj_{\class{B}}$ is $\class{U}$-acyclic. In particular, one can readily see that this is true if $\cat{G}$ admits a set of projective generators.
But we do not need to have enough projectives to discuss objects of finite projective dimension.
We say that an object $A$ of a Grothendieck category $\cat{G}$ has \emph{finite projective dimension} if for any object $B$ there is an integer $n$ for which $\Ext^i_{\cat{G}}(A,B) = 0$ for all $i>n$.  Following~\cite[\S5.1]{gillespie-models-for-hocats-of-injectives}, we then say that $\cat{G}$ is \emph{locally finite dimensional} if it possesses a generating set $\mathcal{U} = \{U_i\}$ for which each $U_i$ has finite projective dimension.

\begin{lemma}\label{lemma-U-acyclicity}
Assume that $\cat{G}$ is a  locally finite dimensional  Grothendieck category with $\mathcal{U} = \{U_i\}$  a set of generators of finite projective dimension.
Then a chain complex $X$ of injectives is exact (i.e. acyclic in the usual sense) if and only if it is $\class{U}$-acyclic in the sense of Definition~\ref{def-A-acyclic}.
In particular,  each $X \in  \exinj_{\class{B}}$ is  automatically $\class{U}$-acyclic for any class of objects $\class{B}$.
\end{lemma}

\begin{proof}
If $X$ is $\class{U}$-acyclic then it is certainly exact, as $\class{U}$ is a set of generators. (This holds for any set of generators.) For the converse we use that our generators are of finite projective dimension. Indeed if $X$ is an exact complex of injectives and $U_i \in \class{U}$, then by dimension shifting we get $\Ext^1_{\cat{G}}(U_i, Z_nX) \cong \Ext^{1+j}_{\cat{G}}(U_i, Z_{n+j}X)$ and this vanishes for large enough $j$.  It follows that $X$ is  $\class{U}$-acyclic.
\end{proof}

With the hypotheses of Lemma~\ref{lemma-U-acyclicity}, if we  take  the set $A$ in Theorem~\ref{them-A-acyclic-model}  to be $\mathcal{U}$, we immediately recover~\cite[Them.~5.7]{gillespie-models-for-hocats-of-injectives}; existence of the injective  model structure for Krause's stable derived category.
See Section~\ref{subsec-stable-derived} for more discussion.

\section{Completeness of Gorenstein injective cotorsion pairs}\label{sec-Goren-inj}

Again, assume throughout that $\cat{G}$ is  a Grothendieck category.
An object $M \in \cat{G}$ is called \emph{Gorenstein injective} if it is equal to the 0-cycle, $M=Z_0X$, of some totally acyclic complex of injectives $X$. The goal of this section is to understand  when  the Gorenstein injectives are the right side of a complete cotorsion pair, and when  it is cogenerated by a set.
Using our methods, we can do this simultaneously for many classes  of
Gorenstein $\class{B}$-injectives in the following general sense.

\begin{definition}\label{def-G-inj-B-tate}
Let $\class{B}\subseteq \class{G}$ be any class of objects. 
  We remind the reader of some notation set in  Section~\ref{section-finite-dim}:
\begin{itemize}
\item $\exinj_{\class{B}}$ denotes the class of all exact $\class{B}$-acyclic complexes of injectives. 
\item $\class{W} := \leftperp{\exinj_{\class{B}}}$ is the left $\Ext$ orthogonal in $\cha{G}$.
\end{itemize}
Now, we set the following definitions and notation:
\begin{itemize}
\item  An object $N \in \cat{G}$ is called \emph{Gorenstein $\class{B}$-injective} if it is equal to the 0-cycle, $N=Z_0X$, of some exact $\class{B}$-acyclic complex of injectives  $X \in \exinj_{\class{B}}$.
\item  $\class{GI}_{\class{B}}$ denotes the class of all Gorenstein $\class{B}$-injective objects.
\item  An object $M\in\cat{G}$ will be called \emph{$\class{B}$-Tate trivial} if $\Hom_{\cat{G}}(M,X)$ is an exact complex of abelian groups for all $X\in\exinj_{\class{B}}$.
\end{itemize}
 \end{definition}

See Remark~\ref{remark-B-tate} and the related Example~\ref{example-Krause} for an explanation  of  our  `$\class{B}$-Tate trivial'  terminology.  The $\class{B}$-Tate trivial objects will indeed be the trivial objects of the Gorenstein injective model structure, as the next lemma indicates.

 \begin{lemma}\label{lemma-B-Tate-char}
In Definition~\ref{def-G-inj-B-tate}, set $\class{W}_{\cat{G}} :=  \class{W} \cap \cat{G}$. More formally, this is the class of objects in $\cat{G}$ defined by $$\class{W}_{\cat{G}} := \{\mkern2mu M  \in \cat{G}  \,|\, S^0(M) \in \class{W} \mkern2mu \}.$$
Then we have  $\class{W}_{\class{G}} = \leftperp{\class{GI}_{\class{B}}}$, and this is exactly the class of all $\class{B}$-Tate trivial objects.
Moreover, $\class{W}_{\class{G}}$ is a thick class.
 \end{lemma}

 \begin{proof}
We have  $M\in \leftperp{\class{GI}_{\class{B}}}$, if and only if,  $\Ext^1_{\cat{G}}(M, N) = 0$ for all $N \in \class{GI}_{\class{B}}$, if and only if,  $\Ext^1_{\cat{G}}(M, Z_0X) = 0$ for all $X \in \exinj_{\class{B}}$.
But there is a well-known isomorphism $$\Ext^1_{\cat{G}}(M, Z_0X) \cong \Ext^1_{\cha{G}}(S^0(M), X)$$
which holds whenever $X$ is an exact chain complex.  (For example, the isomorphism is proved in~\cite[Lemma~4.2]{gillespie-degreewise-model-strucs}.) So the previous conditions happen if and only if $\Ext^1_{\cha{G}}(S^0(M), X) = 0$ for all $X \in \exinj_{\class{B}}$, or, if and only if $S^0(M)\in \leftperp{\exinj_{\class{B}}} =\class{W}$. This is the case, by definition, if and only if, $M\in\class{W}_{\cat{G}}$.
This proves $\class{W}_{\cat{G}} = \leftperp{\class{GI}_{\class{B}}}$.

Next, note that $M\in\leftperp{\class{GI}_{\class{B}}}$ if and only if   $\Ext^1_{\cat{G}}(M, Z_nX) = 0$ for all $X \in  \exinj_{\class{B}}$, and all $n$.  Since $X$ is an exact complex of injectives it is easy to see that this is equivalent to saying  $M$ is  $\class{B}$-Tate trivial in the sense of Definition~\ref{def-G-inj-B-tate}.

For $X\in \exinj_{\class{B}}$, we have $\Ext^{1}_{\cha{G}}(\Sigma^{n}(-),X) \cong H_{n-1}[\Hom_{\cat{G}}(-,X)]$ (as in the proof of Theorem~\ref{them-A-acyclic-model}). It follows that $\class{W}$ and hence $\class{W}_{\cat{G}} = \class{W}\cap\cat{G}$  are both thick.
\end{proof}

\begin{definition}\label{definition-B Tate trivial-generators}
Let $\class{B}\subseteq \cat{G}$ be a class of objects.
If $\class{U}$  is a generating set for $\cat{G}$, and $\class{U}\subseteq \class{W}_{\class{G}}  \, (= \leftperp{\class{GI}_{\class{B}}}$), then we will call it a \emph{set of $\mathcal{B}$-Tate trivial generators}. In the special  case  of when $\class{B}$ is the class of all injectives we just call $\class{U}$ a set of \emph{Tate trivial generators}.
\end{definition}

The next lemma tells us that if  there is any hope for $(\leftperp{\class{GI}_{\class{B}}}, \class{GI}_{\class{B}})$ to be a complete cotorsion pair, then $\cat{G}$ must admit a  set $\class{U}$ of $\class{B}$-Tate trivial generators.
Equivalently, $\cat{G}$ must admit a set of generators $\class{U}$ as we assumed in  Proposition~\ref{them-exact-B-acyclic-models}.

\begin{lemma}\label{lemma-generators-cond}
Let $\class{B}\subseteq \cat{G}$ be a class of objects. The following statements are equivalent for a given set $\mathcal{U}$ of generators for $\cat{G}$. Moreover, the existence of such a generating set is necessary if $(\leftperp{\class{GI}_{\class{B}}}, \class{GI}_{\class{B}})$ is to be a complete cotorsion pair.
\begin{enumerate}
\item $\class{U}$ is a set of $\class{B}$-Tate trivial generators. 
\item Each $X \in  \exinj_{\class{B}}$ is $\class{U}$-acyclic.
\end{enumerate}
Finally, we note that these conditions are automatically satisfied (for any class $\class{B}$)
whenever $\class{U}$ is a set of generators of finite projective dimension.
\end{lemma}

\begin{proof}
The equivalence of the  conditions on $\class{U}$ follows immediately from Lemma~\ref{lemma-B-Tate-char}.
Next, assume  $(\leftperp{\class{GI}_{\class{B}}}, \class{GI}_{\class{B}})$ is a complete cotorsion pair.  Being Grothendieck, there does exist some set,  say  $\{G_i\}$, of generators for $\cat{G}$.  For each $i$, using  that $(\leftperp{\class{GI}_{\class{B}}}, \class{GI}_{\class{B}})$ has enough projectives, write  an epimorphism $W_i \twoheadrightarrow G_i$ with $W_i \in \leftperp{\class{GI}_{\class{B}}}$. Then $\{W_i\}$ must also be a generating set for $\cat{G}$. In particular, $\leftperp{\class{GI}_{\class{B}}}$ contains the generating set $\mathcal{U}:=\{W_i\}$.

The final statement was shown in Lemma~\ref{lemma-U-acyclicity}.
\end{proof}

The case we are interested in is when $\class{B} := \rightperp{\class{S}}$ is the right side of a cotorsion pair $(\class{A},\class{B})$ cogenerated by a set of objects $\class{S}$.
We can now state our main result.

\begin{theorem}\label{them-Goren-inj-cot}
Let $\cat{G}$ be a  Grothendieck category and $(\class{A}, \class{B})$ be a cotorsion pair that is cogenerated by a set. Then
 $(\leftperp{\class{GI}_{\class{B}}}, \class{GI}_{\class{B}})$ is a complete cotorsion pair if and only if  $\cat{G}$ admits a set $\class{U}$  of $\class{B}$-Tate trivial generators. 
 In this case,
 $(\class{W}_{\cat{G}}, \class{GI}_{\class{B}})$ is an injective  cotorsion pair cogenerated by a set. We call it the \textbf{Gorenstein $\class{B}$-injective cotorsion pair}. In particular, $\mathfrak{M} = (All, \class{W}_{\cat{G}}, \class{GI}_{\class{B}})$  represents a  cofibrantly generated and injective abelian model structure on $\cat{G}$. We call it  the \textbf{Gorenstein $\class{B}$-injective model structure} and it satisfies these properties:
\begin{enumerate}
\item $\textnormal{Ho}(\mathfrak{M})$ is well-generated, and the kernel of the triangulated localization functor $\gamma \colon \cat{G} \xrightarrow{} \textnormal{Ho}(\mathfrak{M})$ contains $\mathcal{U}$ and $\class{B}$ and is closed under direct limits.
\item For any other injective model structure $(All, \class{V}, \class{F})$ with $\class{B}\subseteq \class{V}$, we must have $\class{W}_{\cat{G}}\subseteq \class{V}$, equivalently, $\class{F}\subseteq \class{GI}_{\class{B}}$. 
\item The sphere functor $S^0(-) \colon \cat{G} \xrightarrow{} \cha{G}$ is a (left) Quillen equivalence from the Gorenstein $\class{B}$-injective model structure  to the exact $\class{B}$-acyclic model structure of  Corollary~\ref{cor-exact-injectives-B-acyclic}. (See Proposition~\ref{prop-quillen-equivalence}.)
\end{enumerate}
\end{theorem}

Our proof of Theorem~\ref{them-Goren-inj-cot} will appear after we  obtain two easy lemmas that will be used in the proof.

\begin{lemma}\label{lemma-W-perp}
$\class{GI}_{\class{B}} \subseteq \rightperp{\class{W}_{\cat{G}}}$.
\end{lemma}

\begin{proof}
We have $\class{W}_{\cat{G}} = \leftperp{\class{GI}_{\class{B}}}$ by Lemma~\ref{lemma-B-Tate-char}. So
$\class{GI}_{\class{B}} \subseteq \rightperp{(\leftperp{\class{GI}_{\class{B}}})} =  \rightperp{\class{W}_{\cat{G}}}$.
\end{proof}

For the reverse containment, and the `best injective model' claim in property (2) of Theorem~\ref{them-Goren-inj-cot}, we will use:

\begin{lemma}\label{lemma-B-cot-pair}
Let $\class{V}$ be any thick class of $\class{G}$-objects containing $\class{B}$. If $(\class{V}, \rightperp{\class{V}})$ is known to be a complete cotorsion pair, then this is an injective cotorsion pair with the property that $\rightperp{\class{V}}  \subseteq \class{GI}_{\class{B}}$.
\end{lemma}

\begin{proof}
Assume we have a complete cotorsion pair $(\class{V}, \rightperp{\class{V}})$ where $\class{V}$ is some thick class containing $\class{B}$. We note that the cotorsion pair is necessarily hereditary; (this follows from the completeness and the thickness of $\class{V}$. See~\cite[Lemma~2.3]{gillespie-recollement}.) Since $\class{B}$ contains all of the injective objects, so does $\class{V}$.
Then by a standard argument  we have that  $(\class{V}, \rightperp{\class{V}})$ is in fact an injective cotorsion pair. That is,
$\class{V}\cap\rightperp{\class{V}}$ is precisely the class of injective objects; see~\cite[Prop.~3.6(3)]{gillespie-recollement}. Now the argument given in~\cite[Them.~5.2]{gillespie-recollement} generalizes in a straightforward way to show that  each $N \in \rightperp{\class{V}}$ is equal to the 0-cycle of some exact $\class{B}$-acyclic complex of injectives. The argument goes like this: Take a standard injective  coresolution of $N$. By the hereditary condition, each cokernel in the coresolution is again in $\rightperp{\class{V}}$.  To build the resolution the other way, start by taking a special $\class{V}$-precover of $N$.  Then the precover must be an injective object and the kernel again is in $\rightperp{\class{V}}$. So we may iterate this process to get an injective \emph{resolution} of $N$, again with each kernel in $\rightperp{\class{V}}$. Pasting together the coresolution and the resolution we obtain an exact complex of injectives $X$, with $N = Z_0X$, and for which $\Hom_{\cat{G}}(M,X)$ is exact whenever $M \in\class{V}$. We are done since $\class{B}\subseteq \class{V}$.
\end{proof}

\begin{proof}[Proof of Theorem~\ref{them-Goren-inj-cot}]
($\implies$) This is Lemma~\ref{lemma-generators-cond}.

\noindent ($\impliedby$)  Assume we have a set  $\class{U}$ of $\class{B}$-Tate trivial generators, or equivalently, that each $X \in  \exinj_{\class{B}}$ is $\class{U}$-acyclic. Then  
we have the exact $\class{B}$-acyclic  model structure of Corollary~\ref{cor-exact-injectives-B-acyclic},
$$\mathfrak{M}_{\widehat{T}} = (All, \class{W}, \exinj_{\class{B}}).$$
Its construction shows  the cotorsion pair to be cogenerated by a set of chain complexes $\class{S}$. (Explicitly, we have taken  $T$ to be a set as guaranteed by Corollary~\ref{cor-rightperp-accessible}  and then set $\widehat{T} := T\cup\mathcal{U}$. Then our explicit cogenerating set goes back to the proof of
Theorem~\ref{them-A-acyclic-model}, and in this case becomes
\[
\class{S} = \{D^{n} (U_i/C) \,|\, U_i \in\mathcal{U} , \ \forall C \subseteq U_i\} \bigcup   \{S^{n} (A_j) \,|\, A_j \in \widehat{T} \}.)
\]
It follows from~\cite[Remark~1.8 and Prop.~2.9(1)]{stovicek-deconstructible} that there exists a set $\class{S}' \supseteq \class{S}$ such that $\textnormal{Filt}(\class{S}') = \class{W}$. Here, $\textnormal{Filt}(\class{S}')$ denotes the class of all transfinite extensions of objects in $\class{S}'$.
Now let $\class{S}'_{\cat{G}} := \class{S}'\cap\cat{G}$, or more formally
$$\class{S}'_{\cat{G}} := \{\mkern2mu M \in \cat{G} \,|\, S^0(M) \in \class{S}' \mkern2mu \}.$$
 Then
$\class{S}'_{\cat{G}} \subseteq \class{W}_{\cat{G}}$. If $M\in
\class{W}_{\cat{G}}$, then $S^{0}(M) \in \cat{W}$, and so the complex $S^0(M)$ must be
a transfinite extension of objects of $\class{S}'$.  However, each complex
$X_{\alpha}$ in this transfinite extension is necessarily a subobject of
$S^{0}(M)$, and as such must again be of the form $S^{0}(M_{\alpha})$ for some object $M_{\alpha}$.
It follows that $M$ is a transfinite extension of objects in $\class{S}'_{\cat{G}}$, and hence $\class{W}_{\cat{G}} = \textnormal{Filt}(\class{S}'_{\cat{G}})$. So by ~\cite[Cor.~9.4]{gillespie-book} (Eklof's Lemma) we have $\class{W}_{\cat{G}} = \textnormal{Filt}(\class{S}'_{\cat{G}}) \subseteq   \leftperp{(\rightperp{\class{S'_{\cat{G}}}})}$. Since our generating set $\class{U}$ is contained in $\class{W}_{\cat{G}}$, we have that
$(\leftperp{(\rightperp{\class{S'_{\cat{G}}}})},  \rightperp{\class{S'_{\cat{G}}}})$ is a cofibrantly generated cotorsion pair; see~\cite[Def.~9.30 and Remark~9.31]{gillespie-book}. Since we already know that $\class{W}_{\cat{G}} = \textnormal{Filt}(\class{S}'_{\cat{G}})$ is  closed under direct summands, we see from~\cite[Them.~9.34(2)]{gillespie-book} that $$(\leftperp{(\rightperp{\class{S'_{\cat{G}}}})},  \rightperp{\class{S'_{\cat{G}}}}) =  (\class{W}_{\cat{G}}, \rightperp{\class{W}_{\cat{G}}})$$ and that this is a (functorially) complete  cotorsion pair.
 It now follows from Lemmas~\ref{lemma-W-perp} and~\ref{lemma-B-cot-pair} that  $\rightperp{\class{W}_{\cat{G}}} = \class{GI}_{\class{B}}$ and that this is an injective cotorsion pair.

Finally, we comment on the enumerated properties. First, that $\cat{G}$ is well-generated follows from a general result (\cite[Cor.~12.24]{gillespie-book}) and the notion of a triangulated localization is discussed in~\cite[\S6.7]{gillespie-book}.
Property (2) follows from Lemma~\ref{lemma-B-cot-pair}. Finally, the Quillen equivalence stated in (3) will be proved in  Proposition~\ref{prop-quillen-equivalence}.
\end{proof}

In particular, we have  the special case  ($\class{B}$ = all injectives) of the absolute  Gorenstein injectives. We simply denote this class by $\class{GI}$.   

\begin{corollary}\label{cor-Goren-inj}
For any Grothendieck category $\cat{G}$,  we have  that $(\leftperp{\class{GI}}, \class{GI})$ is a complete cotorsion pair and enjoys all of the properties listed in Theorem~\ref{them-Goren-inj-cot} if and only if $\cat{G}$ admits a set of Tate trivial generators.
In particular, this is the case for  any locally finite dimensional Grothendieck category $\cat{G}$.
 \end{corollary}

\section{The Quillen equivalence and Gorenstein injective envelopes}

Continuing our work in Section~\ref{sec-Goren-inj},  let $\cat{G}$ be a  Grothendieck category and $(\class{A}, \class{B})$ be a cotorsion pair that is cogenerated by a set.
We still need to prove part (3) of Theorem~\ref{them-Goren-inj-cot}.  Of course we also wish to state the  existence of  Gorenstein injective envelopes.
So  continue throughout this section to  assume  that  $\cat{G}$ admits a set $\class{U}$ of  $\class{B}$-Tate trivial generators. Thus  $(\class{W}_{\cat{G}}, \class{GI}_{\class{B}})$  is an   injective cotorsion pair.

\subsection{The Quillen equivalence}
We now prove  part (3) of Theorem~\ref{them-Goren-inj-cot}, which asserts that  the Gorenstein $\class{B}$-injective model structure  is Quillen equivalent to the exact $\class{B}$-acyclic model structure of  Corollary~\ref{cor-exact-injectives-B-acyclic}. The main tool for this is the following lemma which relates the trivial objects in the two model structures.

\begin{lemma}\label{lem-cycles-of-W}
 Let $\mathfrak{M}_{\widehat{T}} = (All, \class{W}, \exinj_{\class{B}})$ be the exact $\class{B}$-acyclic model structure of  Corollary~\ref{cor-exact-injectives-B-acyclic}.
Suppose $W$ is a complex with $H_{i}W=0$
for $i<0$ and $W_{i} \in \class{B}$ for $i>0$.  Then $W \in\class{W}$  if and only if $Z_0W \in \class{W}_{\cat{G}}$. I.e., $S^0(Z_0W)\in \class{W}$.
\end{lemma}

\begin{proof}
This is the analog of~\cite[Lemma~5.1]{bravo-gillespie-hovey}, but we give an alternate proof.  
 By properties previously shown,  we have: (i) any disk $D^n(M)$ on any object $M\in\cat{G}$ is in the class $\class{W}$, and (ii) any sphere $S^n(B)$, for any object $B \in \class{B}$, is also in $\class{W}$. From (i) we infer that any bounded above exact complex is trivial, as any such complex is a transfinite extension of disks. From (ii) we infer that any bounded below complex of objects from $\class{B}$ is trivial, as any such complex is a transfinite extension of spheres $S^n(B_n)$ with each $B_n\in\class{B}$.  Now the given chain complex $W$ has a subcomplex $Y \subseteq W$, where $Y$ is the bounded below complex $\cdots \xrightarrow{} W_2 \xrightarrow{} W_1 \xrightarrow{} Z_0W \xrightarrow{} 0$. Then note that $W/Y$ is (isomorphic to) the complex $0 \xrightarrow{} Z_{-1}W \xrightarrow{} W_{-1} \xrightarrow{} W_{-2} \xrightarrow{} \cdots$, which is bounded above and exact and so is trivial in the exact $\class{B}$-acyclic model structure. Therefore, since the class of trivial complexes is thick, $W$ is trivial if and only if $Y$ is trivial. But we have another subcomplex, $S^0(Z_0W) \subseteq Y$, whose quotient is a bounded below complex of objects from $\class{B}$. Thus $Y$ (and hence $W$) is trivial if and only if  $S^0(Z_0W)$ is trivial. That is,  if and only if $Z_0W \in \class{W}_{\cat{G}}$.
\end{proof}

\begin{proposition}\label{prop-quillen-equivalence}
The sphere functor $S^0(-) \colon \cat{G} \xrightarrow{} \cha{G}$ is a (left) Quillen equivalence from the Gorenstein $\class{B}$-injective model structure of Theorem~\ref{them-Goren-inj-cot} to the exact $\class{B}$-acyclic model structure of  Corollary~\ref{cor-exact-injectives-B-acyclic}. Its right adjoint (Quillen inverse) is given by the 0-cycles functor, $Z_0(-)$.
Therefore, the associated total derived adjunction is a triangulated equivalence between their homotopy categories.
\end{proposition}

\begin{proof}
That $(S^0, Z_0)$ is an adjoint pair is standard; see~\cite[Lemma~10.1(3)]{gillespie-book}. The left adjoint $S^0$ is clearly  exact and preserves trivial objects. It follows that $S^0$ preserves cofibrations and trivial cofibrations,  so $(S^0, Z_0)$ is indeed a Quillen adjunction~\cite[Prop.~7.7]{gillespie-book}.
To verify that this is a Quillen equivalence we will check the conditions in~\cite[Them~7.17(2)]{gillespie-book}. These conditions involve the unit and counit. We note that the unit  of the adjunction  $(S^0, Z_0)$  is merely the identity map $\eta_M \colon M \xrightarrow{} Z_0S^0(M)$. The counit is the evident inclusion $\epsilon_X \colon S^0(Z_0X) \xrightarrow{} X$.
 Now for an arbitrary (cofibrant) object $M\in\cat{G}$, we find a short exact sequence
\[
0 \xrightarrow{} S^{0}(M) \xrightarrow{} X \xrightarrow{}W \xrightarrow{} 0
\]
 with $X \in \exinj_{\class{B}}$ and $W\in\class{W}$.  By applying the snake lemma (in degrees $0$ and $-1$) we get a short exact sequence
\[
0 \xrightarrow{} M \xrightarrow{} Z_{0}X\xrightarrow{} Z_{0}W \xrightarrow{} 0.
\]
It follows from Lemma~\ref{lem-cycles-of-W} that $Z_{0}W\in\class{W}$, because the $W_{i} \in \class{B}$ are injective for all $i\neq 0$ and $H_{i}W=0$ for all $i\neq 1$.  It follows that $M \rightarrowtail Z_0X$ is a trivial cofibration. This verifies the first condition (involving the morphism $j^{\#}$) in~\cite[Them~7.17(2)]{gillespie-book}. We now turn to the second condition (the one involving the morphism $p^{\flat}$). So let $X \in  \exinj_{\class{B}}$ be an arbitrary fibrant object. Since $Z_0X$ is already cofibrant, the map $p^{\flat}$ is just the evident inclusion  $\epsilon_X \colon S^0(Z_0X) \xrightarrow{} X$ that is the counit of the adjunction. We wish to show that it too is a weak equivalence (trivial cofibration) and for this it suffices to show that its cokernel is in $\class{W}$. But this again follows from Lemma~\ref{lem-cycles-of-W}. Indeed  its cokernel is the complex
$$\cdots \xrightarrow{d_3}X_2  \xrightarrow{d_2}X_1  \xrightarrow{\pi d_1}X_0/Z_0X  \xrightarrow{\bar{d}_0}X_{-1}  \xrightarrow{d_{-1}}X_{-2} \xrightarrow{} \cdots $$
where $\bar{d}_0$ is the unique map satisfying $d_0 = \bar{d}_0\pi$. Since $X$ is an exact complex of injectives, this complex satisfies the hypotheses of Lemma~\ref{lem-cycles-of-W}, and so it is in $\class{W}$ if and only if its 0-cycle is in $\class{W}_{\cat{G}}$ (trivial in the Gorenstein injective model structure on $\cat{G}$). But  the 0-cycle object of the above complex is clearly $\Ker{\bar{d}_0} = 0$, which is certainly in $\class{W}_{\cat{G}}$. This completes the proof that that $p^{\flat}$ is a trivial cofibration, and  that $(S^0, Z_0)$ is a Quillen equivalence.

The final statement follows from general  results; see~\cite[Thems.~7.12 and~7.17]{gillespie-book}.
\end{proof}

\subsection{Remark on our terminology of `$\class{B}$-Tate trivial' objects}~\label{remark-B-tate}
 Given any object $M\in\cat{G}$, we may write a short exact sequence
\[
0 \xrightarrow{} S^{0}(M) \xrightarrow{} X^M \xrightarrow{}W^M \xrightarrow{} 0
\]
 with $X^M \in \exinj_{\class{B}}$ and $W^M\in\class{W}$. As shown in the proof of Proposition~\ref{prop-quillen-equivalence},  by applying $Z_0(-)$ we get a short exact sequence
 \[
0 \xrightarrow{} M \xrightarrow{} Z_0X^M \xrightarrow{} Z_0W^M \xrightarrow{} 0
\]
 with $Z_0X^M \in \class{GI}_{\class{B}}$ and $Z_0W^M\in\class{W}_{\class{G}}$. Given  another object $A\in\cat{G}$, we may define the  \emph{$\class{B}$-Tate cohomology} groups by
 $${}_{\class{B}}\widehat{\Ext}_{\cat{G}}^n(A,M) := H_{-n}[\Hom_{\cat{G}}(A,X^M)].\footnote{The discussion in~\cite[\S9.1]{estrada-gillespie-quillen-duality} gives further details by  describing the connection to~\cite[Def.~7.5]{krause-stable-derived} and  general theory developed in~\cite{gillespie-canonical resolutions}.}$$
Then $A$ is $\class{B}$-Tate trivial in the sense of our Definition~\ref{def-G-inj-B-tate} if and only if we have ${}_{\class{B}}\widehat{\Ext}_{\cat{G}}^n(A,M) = 0$  for all $M$ and for all $n\in \Z$.  The same is  true in the other variable. In particular, if $A$ is $\class{B}$-Tate trivial,  then its corresponding complex $X^A$ in the above short exact sequence must be a contractible complex of injectives. So we again have ${}_{\class{B}}\widehat{\Ext}_{\cat{G}}^n(M,A) = 0$ for all $M$ and for all $n\in \Z$. These observations are in line with Krause's comments in the paragraph before his~\cite[Them.~1.3]{krause-stable-derived}. See also Example~\ref{example-Krause}.

\subsection{Gorenstein injective envelopes}
Here we point out two corollaries to Theorem~\ref{them-Goren-inj-cot}. The first is that every object has a Gorenstein injective envelope.  

\begin{corollary}\label{cor-Goren-inj-cot}
 The Gorenstein $\class{B}$-injective cotorsion pair $(\class{W}_{\cat{G}}, \class{GI}_{\class{B}})$  of Theorem~\ref{them-Goren-inj-cot} is  perfect. That is,  the class $\class{W}_{\cat{G}}$ is covering and the class $\class{GI}_{\class{B}}$ is enveloping.
\end{corollary}

\begin{proof}
The class $\class{W}_{\cat{G}}$ is closed under direct limits. So the corollary follows from  a classical argument due to Enochs that can be found in \cite[\S2.2]{xu_flat_covers}. Although the arguments there are given for module categories, the results remain valid for Grothendieck categories (cf. \cite[\S2]{enochs2001}).
In particular, note that any special $\class{GI}_{\class{B}}$-preenvelope of an object $M$ is, in the terminology of \cite[Def.~2.2.1]{xu_flat_covers},  a \emph{generator for $\Ext(\class{W}_{\cat{G}}, M)$}. Then the existence of Gorenstein injective envelopes follows from~\cite[Prop.~2.2.1/Them.~2.2.2]{xu_flat_covers}. Likewise, the existence of $\class{W}_{\cat{G}}$-covers follows from~\cite[Them.~2.2.12]{xu_flat_covers}.
\end{proof}

Now let $\textnormal{St}(\cat{G})$ denote the \emph{injective stable category} of $\cat{G}$. Its objects are  the same  as those of $\cat{G}$.  Its morphism sets are defined by
$$\underline{\Hom}(M,N) := \Hom_{\cat{G}}(M,N)/\sim$$
where $f \sim g$ if and only if $g-f$ factors through an injective object.   

It follows immediately from Theorem~\ref{them-Goren-inj-cot} that $\class{GI}_{\class{B}}$ is a Frobenius category whose projective-injective objects are precisely the usual injectives. Moreover, its stable category identifies with  the homotopy category of the Gorenstein $\class{B}$-injective model structure. Indeed there is a  triangle equivalence
 $$\textnormal{Ho}(\mathfrak{M}) \simeq \textnormal{St}(\class{GI}_{\class{B}})\subseteq \textnormal{St}(\cat{G}).$$
 (Each of the above claims follows from standard results. See~\cite[\S8.1/\S8.2]{gillespie-book}.)
We now have the following result along the lines of~\cite[Corollary~5.13]{cortez-crivei-saorin-reflective}.

\begin{corollary}
$\textnormal{St}(\class{GI}_{\class{B}})\subseteq \textnormal{St}(\cat{G})$ is a reflective subcategory. Indeed, Gorenstein injective approximations (i.e. special preenvelopes)  determine a functor $$R \mathcolon \textnormal{St}(\cat{G}) \xrightarrow{}  \textnormal{St}(\class{GI}_{\class{B}})$$ which  is left adjoint to the inclusion functor $\textnormal{St}(\class{GI}_{\class{B}})\hookrightarrow \textnormal{St}(\cat{G})$.
\end{corollary}

\begin{proof}
This follows from~\cite[Prop.~3.10]{gillespie-book}. In particular,  take  $\lambda$ (see~\cite[Setup 3.2.2 on page~63]{gillespie-book} to be $\lambda = \omega := \class{W}_{\cat{G}} \cap \class{GI}_{\class{B}}$, which is exactly the class of injective objects.
\end{proof}

\section{Examples and applications}\label{section-Examples-and-applications}

Now we briefly discuss the key examples that motivated our work, as well as  some unexpected examples and applications.

\subsection{Gorenstein injectives}
By Corollary~\ref{cor-Goren-inj},   for any locally finite dimensional Grothendieck category $\cat{G}$,  we obtain the Gorenstein injective model structure. We also  have  Gorenstein injective envelopes from Corollary~\ref{cor-Goren-inj-cot}. Our  main  new  application  of this  is described in the next example.

\begin{example}\label{qco}
Let $\mathbb{X}$ be a quasi-compact and semi-separated scheme, with $\{U_i\}_{i=1}^N$ a semiseparating open affine covering of $\mathbb{X}$. We let $\qcox$ denote the category of quasi-coherent sheaves on $\mathbb{X}$. Then $\qcox$ admits a family of \emph{very flat} generators, as introduced by Positselski in \cite{positselski-contraherent-cosheaves}, and all such sheaves have finite projective dimension not exceeding $N$.
In more detail, by \cite[Cor.~4.1.2 (b), Cor.~4.1.4 and Lemma B.1.2]{positselski-contraherent-cosheaves}, the pair of classes $(\mathcal{VF}(\mathbb{X}), \mathcal{CA}(\mathbb{X}))$ of very flat and \emph{contraadjusted} quasi-coherent sheaves on $\mathbb{X}$ forms a complete hereditary cotorsion pair on $\qcox$. In particular, $\mathcal{VF}(\mathbb{X})$ contains a set of generators for $\qcox$. Now, by \cite[Lemma~5.5.12(c)]{positselski-contraherent-cosheaves}, together with the equivalent conditions (3) and (4) of \cite[Lemma B.1.9(c)]{positselski-contraherent-cosheaves}, it follows that every very flat quasi-coherent sheaf has finite projective dimension not exceeding $N$.

Therefore,  the Gorenstein injective model structure exists on $\qcox$, and every quasi-coherent sheaf has a Gorenstein injective envelope. We note that the Quillen equivalence of Proposition~\ref{prop-quillen-equivalence} strengthens our previous result  from~\cite[Them.~5.2]{estrada-gillespie-quillen-duality}, which was for semi-separated Noetherian schemes.
\end{example}

\begin{example}\label{example-Krause}
Let $\cat{G}$ be a locally noetherian Grothendieck category whose derived category is compactly generated. In \cite[Them.~7.12]{krause-stable-derived}, Krause shows that  the objects with vanishing Tate cohomology, along  with the Gorenstein injectives,   form a complete cotorsion pair in $\cat{G}$. In particular, the compact generation implies the existence of Tate trivial generators for $\cat{G}$.
\end{example}

\subsection{Ding injectives, and Gorenstein $\mathbf{FP_n}$-injectives}
The pursuit of Gorenstein injective envelopes of modules over a ring led to a number of interesting variants. These include the Ding injective modules,  the Gorenstein AC-injective modules, and the Gorenstein $\operatorname{FP}_n$-injective modules; see \cite{bravo-gillespie-hovey}, \cite{gillespie-ding-modules}, and \cite{bravo-gillespie-perez-fpn}, and further references therein. It was shown in~\cite{gillespie-iacob-ding-injective-envelopes} that over any ring, envelopes exist for each of these types of modules.
Just  like the case of Gorenstein injectives this can be generalized to Grothendieck categories.

 \begin{example}\label{example-FP_n}
For any Grothendieck category $\cat{G}$,  
 there is a set $\class{S}$ of isomorphism representatives for all objects of type $\operatorname{FP}_n$ in the sense of~\cite{bravo-gillespie-perez-fpn}. (This fact is nicely explained in~\cite[\S2.2]{bravo-odabasi-parra-perez-fpn}.)
Such a set cogenerates the class $\class{B} :=\rightperp{\class{S}}$, of all $\operatorname{FP}_n$-injective objects. So whenever $\cat{G}$ is locally finite dimensional we obtain Gorenstein $\operatorname{FP}_n$-injective envelopes and model structures. It is notable that we previously only knew this  in the  case that $\cat{G}$ is, in addition, an  \emph{$n$-coherent} category in the sense of~\cite{bravo-gillespie-perez-fpn}.  (For example,  Gorenstein injectives when $\cat{G}$  is locally noetherian,   Ding injectives when $\cat{G}$ is locally coherent, etc. See  \cite[Remark~6.3]{bravo-gillespie-perez-fpn}.)
\end{example}

\subsection{Strongly FP-injectives, and other examples}
Consider again, properties (1) and (2) stated in Theorem~\ref{them-Goren-inj-cot}. For any  $\class{B} :=\rightperp{\class{S}}$, where $\class{S}$ is a given set,   the Gorenstein $\class{B}$-injective model structure is the best possible injective model structure  that `kills'   the class $\class{B}$.
Here we point out just one interesting example of a class $\class{B}$ for which we were previously unaware that such a construction could be made.

\begin{example}
Following \cite{li-guan-ouyang-strongly-fp-injective}, a module $N$ over a ring $R$ is said to be \emph{strongly FP-injective}  if $\Ext^i_R(F,N) = 0$ for all finitely presented modules $F$ and all $i\geq 1$. A number of authors have studied these modules. For example, see~\cite{bazzoni-hrbek-positselski-fp-projective periodicity, emmanouil-kaperonis-k-absolutely-pure}. Let $\class{SFI}$ denote the class of all strongly FP-injectives. Choose a set $\class{S} = \{\Omega^n(F)\}$ of all syzygies of all the modules in some set of isomorphism representatives for the finitely presented modules. Then one can check that we have  $\rightperp{\class{S}} = \class{SFI}$.  Thus we obtain Gorenstein $\class{SFI}$-injective envelopes and model structures for any ring $R$.
\end{example}

\subsection{Generators for the stable derived category}\label{subsec-stable-derived}
When it comes to the problem of finding Tate trivial generating sets $\class{U}$, it is natural to  consider $\exinj$, the class  of all exact complexes of injectives. After all,  if each $X\in\exinj$ is $\class{U}$-acyclic then so too is each $X\in\exinj_{\class{B}}$, for any class $\class{B}$. Note that if we take $\class{P}$ to be the class of all projective objects, then  $\exinj = \exinj_{\class{P}}$  (even if 0 is the only existing projective). 
 So in the terminology of Definition~\ref{definition-B Tate trivial-generators}, saying that $\class{U}$ is a set of $\class{P}$-Tate trivial generators   is equivalent to saying that each $X\in \exinj$ is $\class{U}$-acyclic. Of course, a set of generators of finite projective dimension is such a set;  see Lemma \ref{lemma-U-acyclicity}. In this way we have the following extension of~\cite[Them.~5.7]{gillespie-models-for-hocats-of-injectives}.

\begin{example}\label{ex-P-tate}
If $\cat{G}$ admits a set $\cat{U}$ of $\class{P}$-Tate trivial generators, then taking $A = \class{U} = \class{B}$ in Theorem~\ref{them-A-acyclic-model} immediately yields $\mathfrak{M}_{\class{U}} = (All, \class{W}, \exinj)$. This is the injective model structure for Krause's stable derived category of $\cat{G}$. The related recollement also holds in the same way as~\cite[Cor.~5.8]{gillespie-models-for-hocats-of-injectives}.
\end{example}

We do not know an example, but it seems feasible for a category to  admit a $\class{P}$-Tate trivial set of generators which are not  of finite projective dimension. For example, if $\cat{G}$ has a closed symmetric monoidal structure $\otimes$, then having a set of $\otimes$-flat generators, along with having cotorsion periodicity (see \cite{bazzoni-cortes-estrada-periodic} for a reference in module categories), might provide such an example. This approach was taken in~\cite[Theorem~4.1]{estrada-gillespie-quillen-duality}, though as explained in Example~\ref{qco},   we do have generators of finite projective dimension for such schemes $\mathbb{X}$ as considered in~\cite[Theorem~4.1]{estrada-gillespie-quillen-duality}.

On the other hand,  Neeman's example~\cite[Example~4.1]{neeman-homotopy category of injectives} shows that there exist locally noetherian Grothendieck categories  which do not admit a set of $\class{P}$-Tate trivial generators.
In particular, Grothendieck categories need not admit $\class{B}$-Tate trivial generators for every class $\class{B}$. But we don't know whether or not  this is the case  when   $\class{B}$ is   the class of all injectives.  One would like an example similar to Neeman's but for totally acyclic complexes of injectives.

\subsection{PGF and PGF$_{\class{B}}$ modules} We just remark that the accessibility of right cotorsion classes  can also be used to approach the PGF-modules introduced in~\cite{saroch-stovicek-singular-compactness}. Given a ring $R$, a \emph{projectively coresolved Gorenstein flat} $R$-module $M$, or \emph{PGF-module} for short, is a 0-cycle, $M=Z_0X$, of some exact complex $X$ of projective $R$-modules with the property that $I\otimes_R X$ is an exact complex of abelian groups for all injective (right) $R$-modules. Let $T$ be a set of injective (right) $R$-modules as in Corollary~\ref{cor-injectives-accessible}. Set $J := (\oplus T_i) \oplus R_R $, where $\oplus T_i$ is the coproduct of all the $T_i\in T$.  Then a complex $X$ of projectives has the above property if and only if $J \otimes_R X$ is exact. From this it follows immediately from~\cite[Them.~6.1]{bravo-gillespie-hovey} that we have a cofibrantly generated projective model structure
$$\mathfrak{M} = (\class{C}, \class{W}, All)$$ on $\ch$, where the complexes in $\class{C}$ are precisely the exact complexes $X$ of projectives for which $I \otimes_R X$ remains exact for any injective $I$.

More generally, by Corollary~\ref{cor-rightperp-accessible} we get the analogous result   based on any right cotorsion class, $\class{B}:=\rightperp{\class{S}}$, where $\class{S}$ is a set of (right) $R$-modules.  That is, by a \emph{PGF$_{\class{B}}$-module} we mean a module $M=Z_0X$, for some exact complex $X$ of projective $R$-modules with the property that $B\otimes_R X$ is an exact complex  for all $B\in \class{B}$. Then  such complexes $X$ are the cofibrant objects of an abelian model structure analogous to  $\mathfrak{M}$ above.

\section*{Acknowledgements}
This work was initiated during a visit of the first author to Ramapo College. He would like to thank his coauthor James Gillespie for his hospitality and the excellent working conditions provided during his stay.

\providecommand{\bysame}{\leavevmode\hbox to3em{\hrulefill}\thinspace}
\providecommand{\MR}{\relax\ifhmode\unskip\space\fi MR }
\providecommand{\MRhref}[2]{%
  \href{http://www.ams.org/mathscinet-getitem?mr=#1}{#2}
}
\providecommand{\href}[2]{#2}

\end{document}